\documentclass{amsart}
\usepackage{amsfonts}
\usepackage{mathrsfs}
\usepackage{amsthm}
\usepackage{amssymb}
\usepackage{amsmath}
\usepackage{enumerate}

\theoremstyle{plain}
\newtheorem{theorem}{Theorem}[section]
\newtheorem{proposition}[theorem]{Proposition}

\newtheorem{lemma}[theorem]{Lemma}

\theoremstyle{definition}
\newtheorem{definition}{Definition}[section]

\theoremstyle{remark}
\newtheorem{remark}{Remark}[section]

\theoremstyle{example}
\newtheorem{example}{Example}[section]

\numberwithin{equation}{section}

\begin{document}

\title[Clark-Ocone Formula]
{Clark-Ocone Formula for Generalized Functionals of Discrete-Time Normal Noises}

\author{Caishi Wang}
\address[Caishi Wang]
          {School of Mathematics and Statistics,
          Northwest Normal University,
          Lanzhou, Gansu 730070,
          People's Republic of China}
\email{wangcs@nwnu.edu.cn}

\author{Shuai Lin}
\address[Shuai Lin]
         {School of Mathematics and Statistics,
         Northwest Normal University,
          Lanzhou, Gansu 730070,
          People's Republic of China}

\author{Ailing Huang}
\address[Ailing Huang]
         {School of Mathematics and Statistics,
         Northwest Normal University,
          Lanzhou, Gansu 730070,
          People's Republic of China}


\begin{abstract}
The Clark-Ocone formula in the theory of discrete-time chaotic calculus holds only for square integrable functionals of
discrete-time normal noises. In this paper, we aim at extending this formula to generalized functionals of discrete-time normal noises.
Let $Z$ be a discrete-time normal noise that has the chaotic representation property.
We first prove a result concerning the regularity of generalized functionals of $Z$.
Then, we use the Fock transform to define some fundamental operators on generalized functionals of $Z$,
and apply the above mentioned regularity result to prove the continuity of these operators.
Finally, we establish the Clark-Ocone formula for generalized functionals of $Z$,
and show its application results, which include the covariant identity result
and the variant upper bound result for generalized functionals of $Z$.
\end{abstract}

\maketitle

\section{Introduction}\label{sec-1}

One of the important theorems in Privault's discrete-time chaotic calculus \cite{privault, wang-lc} is its Clark-Ocone formula, which reads
\begin{equation}\label{eq-1-1}
  \xi = \mathbb{E}\xi + \sum_{k=0}^{\infty}Z_k \mathbb{E}[\partial_k \xi | \mathcal{F}_{k-1}],\quad \xi \in \mathcal{L}^2(Z),
\end{equation}
where $Z=(Z_k)$ is a discrete-time normal noise, $\mathcal{L}^2(Z)$ the space of square integrable functionals of $Z$,
$\mathcal{F}_k$ the $\sigma$-field generated by $(Z_j; 0\leq j \leq k)$,
$\partial_k$ the annihilation operator on $\mathcal{L}^2(Z)$,
and the series on the righthand side converges in the norm of $\mathcal{L}^2(Z)$.

The Clark-Ocone formula (\ref{eq-1-1}) directly gives the predictable representation of functionals of $Z$,
which implies the predictable representation property of discrete-time martingales associated with $Z$.
The formula can also be used to establish the corresponding covariant identities \cite{privault}.
More importantly, as was shown by Gao and Privault \cite{gao}, this formula plays an important role in proving logarithmic Sobolev inequalities
for Bernoulli measures. There are other applications based on the formula \cite{wang-lc,wang-z}.

Despite its multiple uses, however, the Clark-Ocone formula (\ref{eq-1-1}) still suffers from a main drawback.
That is, it holds only for the square integrable functionals $\xi$ of $Z$, which excludes many other interesting functionals of $Z$.

On the other hand, as is shown in \cite{wang-z},
one can use the canonical orthonormal basis of $\mathcal{L}^2(Z)$ to construct a nuclear space $\mathcal{S}(Z)$
such that $\mathcal{S}(Z)$ is densely contained in $\mathcal{L}^2(Z)$. Thus, by identifying $\mathcal{L}^2(Z)$ with its dual,
one can get a Gel'fand triple
\begin{equation}\label{eq-1-2}
\mathcal{S}(Z)\subset \mathcal{L}^2(Z) \subset \mathcal{S}^*(Z),
\end{equation}
where $\mathcal{S}^*(Z)$ is the dual of $\mathcal{S}(Z)$, which is endowed with the strong topology,
which can not be induced by any norm \cite{gelfand}.
As usual, $\mathcal{S}(Z)$ is called the testing functional space of $Z$,
while $\mathcal{S}^*(Z)$ is called the generalized functional space of $Z$.
It turns out \cite{wang-chen-1} that the generalized functional space $\mathcal{S}^*(Z)$ can accommodate
many quantities of theoretical interest that can not be covered by $\mathcal{L}^2(Z)$.

In this paper, we would like to extend the Clark-Ocone formula (\ref{eq-1-1}) to the generalized functionals of $Z$.
More precisely, we would like to establish a Clark-Ocone formula for all elements of $\mathcal{S}^*(Z)$.
Our main work is as follows.

We first prove a result concerning the regularity of generalized functionals in $\mathcal{S}^*(Z)$ in Section~\ref{sec-2}.
Then, in Section~\ref{sec-3}, we use the Fock transform \cite{wang-chen-1} to define some fundamental operators on
$\mathcal{S}^*(Z)$, and apply the above mentioned regularity result to prove the continuity of these operators.
Finally, we establish our formula, namely the Clark-Ocone formula for generalized functionals in $\mathcal{S}^*(Z)$ in Section~\ref{sec-3},
and show its application results in Section~\ref{sec-4}, which include the covariant identity result
and the variant upper bound result for generalized functionals in $\mathcal{S}^*(Z)$.

Throughout this paper, $\mathbb{N}$ designates the set of all nonnegative integers and $\Gamma$
the finite power set of $\mathbb{N}$, namely
\begin{equation*}
\Gamma = \{\,\sigma \mid \text{$\sigma \subset \mathbb{N}$ and $\#(\sigma) < \infty$} \,\},
\end{equation*}
where $\#(\sigma)$ means the cardinality of $\sigma$ as a set. If $k\in \mathbb{N}$ and $\sigma \in \Gamma$,
then we simply write $\sigma \cup k$ for $\sigma \cup \{k\}$. Similarly, we use $\sigma \setminus k$.

\section{Generalized functionals of discrete-time normal noises}\label{sec-2}

In all the following sections, we always assume that $(\Omega, \mathcal{F}, P)$ is a given probability space.
We use $\mathbb{E}$ to mean the expectation with respect to $P$.
As usual, $\mathcal{L}^{2}(\Omega, \mathcal{F}, P)$ denotes the Hilbert space of
square integrable complex-valued measurable functions on $(\Omega, \mathcal{F}, P)$.
We use $\langle\cdot,\cdot\rangle$ and $\|\cdot\|$ to mean the inner product and norm of $\mathcal{L}^{2}(\Omega, \mathcal{F}, P)$, respectively.
By convention, $\langle\cdot,\cdot\rangle$ is conjugate-linear in its first argument and linear in its second argument.

\subsection{Discrete-time normal noises}\label{subsec-2-1}

A sequence $Z=(Z_n)_{n\in \mathbb{N}}$ of integrable random variables on
$(\Omega, \mathcal{F}, P)$ is called a discrete-time normal noise if it satisfies:
\begin{enumerate}[(i)]
  \item $\mathbb{E}[Z_n | \mathcal{F}_{n-1}] = 0$
 for $n\geq 0$;
  \item $\mathbb{E}[Z_n^2 | \mathcal{F}_{n-1}] = 1$
for $n\geq 0$.
\end{enumerate}
Here $\mathcal{F}_{-1}=\{\emptyset, \Omega\}$,
$\mathcal{F}_n = \sigma(Z_k; 0\leq k \leq n)$ for $n \in \mathbb{N}$ and $\mathbb{E}[\cdot | \mathcal{F}_{n}]$ means the conditional expectation
given $\mathcal{F}_{n}$.

\begin{example}\label{examp-2-1}
Let $\zeta=(\zeta_n)_{n\in \mathbb{N}}$ be an independent sequence of random variables on $(\Omega, \mathcal{F}, P)$ with
\begin{equation*}
  P\left\{\zeta_n = -1\right\}= P\left\{\zeta_n = 1\right\} = \frac{1}{2},\quad n \in \mathbb{N}.
\end{equation*}
Write $\mathcal{G}_{-1}=\{\emptyset, \Omega\}$ and $\mathcal{G}_n = \sigma(\zeta_k; 0\leq k \leq n)$ for $n \in \mathbb{N}$. Then,
one can immediately see that
\begin{enumerate}[(i)]
  \item $\mathbb{E}[\zeta_n | \mathcal{G}_{n-1}] = 0$
 for $n\geq 0$;
  \item $\mathbb{E}[\zeta_n^2 | \mathcal{G}_{n-1}] = 1$
for $n\geq 0$.
\end{enumerate}
Thus $\zeta$ is a discrete-time normal noise. Note that, by letting $X=(X_n)$ be the partial sum sequence of $\zeta$,
one gets the classical random walk.
\end{example}

For a discrete-time normal noise $Z=(Z_n)_{n\in \mathbb{N}}$ on $(\Omega, \mathcal{F}, P)$,
one can construct a corresponding family $\{Z_{\sigma}\mid \sigma \in \Gamma\}$ of random variables on $(\Omega, \mathcal{F}, P)$
in the following manner
\begin{equation}\label{eq-2-1}
 Z_{\emptyset}=1\quad \text{and}\quad   Z_{\sigma} = \prod_{i\in \sigma}Z_i,\quad \text{$\sigma \in \Gamma$, $\sigma \neq \emptyset$}.
\end{equation}
We call $\{Z_{\sigma}\mid \sigma \in \Gamma\}$ the canonical functional system of $Z$.

\begin{lemma}\cite{emery,privault,wang-lc}\label{lem-2-1}
Let $Z=(Z_n)_{n\in \mathbb{N}}$ be a discrete-time normal noise on $(\Omega, \mathcal{F}, P)$.
Then its canonical functional system $\{Z_{\sigma}\mid \sigma \in \Gamma\}$ forms a countable orthonormal system in
$\mathcal{L}^2(\Omega, \mathcal{F}, P)$.
\end{lemma}

Let $\mathcal{F}_{\infty}=\sigma(Z_n; n\in \mathbb{N})$ be the $\sigma$-field over $\Omega$ generated by
a discrete-time normal noise $Z=(Z_n)_{n\in \mathbb{N}}$ on $(\Omega, \mathcal{F}, P)$.
Then the canonical functional system $\{Z_{\sigma}\mid \sigma \in \Gamma\}$ is also a countable orthonormal system in
the space $\mathcal{L}^2(\Omega, \mathcal{F}_{\infty}, P)$
of square integrable complex-valued measurable functions on $(\Omega, \mathcal{F}_{\infty}, P)$.

In the literature, $\mathcal{F}_{\infty}$-measurable functions on $\Omega$ are also known as functionals of $Z$.
Thus elements of $\mathcal{L}^2(\Omega, \mathcal{F}_{\infty}, P)$ are naturally called square integrable functionals of $Z$.

\begin{definition}\label{def-2-1}
A discrete-time normal noise $Z=(Z_n)_{n\in \mathbb{N}}$ on $(\Omega, \mathcal{F}, P)$ is said to have the chaotic representation property if
its canonical functional system $\{Z_{\sigma}\mid \sigma \in \Gamma\}$ is total in
$\mathcal{L}^2(\Omega, \mathcal{F}_{\infty}, P)$,
where $\mathcal{F}_{\infty}=\sigma(Z_n; n\in \mathbb{N})$.
\end{definition}

Thus, if a discrete-time normal noise $Z$ has the chaotic representation property,
then its canonical functional system $\{Z_{\sigma}\mid \sigma \in \Gamma\}$
is actually an orthonormal basis of $\mathcal{L}^2(\Omega, \mathcal{F}_{\infty}, P)$.

\subsection{Generalized functionals}\label{subsec-2-2}

From now on, we always assume that $Z=(Z_n)_{n\in \mathbb{N}}$ is a given discrete-time normal noise
on $(\Omega, \mathcal{F}, P)$ that has the chaotic representation property.

For brevity, we use $\mathcal{L}^2(Z)$ to denote the space of square integrable functionals of $Z$, namely
\begin{equation*}
  \mathcal{L}^2(Z) = \mathcal{L}^2(\Omega, \mathcal{F}_{\infty}, P),
\end{equation*}
where $\mathcal{F}_{\infty}=\sigma(Z_n; n\in \mathbb{N})$. For $k\geq 0$, we denote by $\mathcal{F}_k$
the $\sigma$-field generated by $(Z_j; 0\leq j \leq k)$, namely
\begin{equation*}
\mathcal{F}_k=\sigma(Z_j; 0\leq j \leq k).
\end{equation*}
We note that $\mathcal{L}^2(Z)$ shares the same inner product $\langle\cdot,\cdot\rangle$ and norm $\|\cdot\|$
 with $\mathcal{L}^2(\Omega, \mathcal{F}, P)$, and moreover the canonical functional system
$\{Z_{\sigma}\mid \sigma \in \Gamma\}$ of $Z$ forms a countable orthonormal basis for $\mathcal{L}^2(Z)$,
which we call the canonical orthonormal basis of $\mathcal{L}^2(Z)$.

\begin{lemma}\label{lem-2-2}\cite{wang-z}
Let $\sigma\mapsto\lambda_{\sigma}$ be the $\mathbb{N}$-valued function on $\Gamma$ given by
\begin{equation}\label{eq-2-2}
\lambda_{\sigma}=
\left\{
  \begin{array}{ll}
    \prod_{k\in\sigma}(k+1), & \hbox{$\sigma\neq \emptyset$, $\sigma\in\Gamma$;}\\
    1, & \hbox{$\sigma=\emptyset$, $\sigma\in\Gamma$.}
  \end{array}
\right.
\end{equation}
Then, for $p>1$, the positive term series $\sum_{\sigma\in\Gamma}\lambda^{-p}_{\sigma}$ converges and moreover
\begin{equation}\label{eq-2-3}
\sum_{\sigma\in\Gamma}\lambda^{-p}_{\sigma}\leq \exp\bigg[\sum_{k=1}^{\infty}k^{-p}\bigg]<\infty.
\end{equation}
\end{lemma}

Using the $\mathbb{N}$-valued function defined by (\ref{eq-2-2}), we can construct a chain of Hilbert spaces consisting of functionals of $Z$ as follows.
For $p\geq 0$, we put
\begin{equation}\label{eq-2-4}
  \mathcal{S}_p(Z) = \Big\{\, \xi \in \mathcal{L}^2(Z) \Bigm| \sum_{\sigma\in \Gamma}\lambda_{\sigma}^{2p}|\langle Z_{\sigma}, \xi\rangle|^{2}< \infty\,\Big\}
\end{equation}
and define
\begin{equation}\label{eq-2-5}
  \langle \xi,\eta\rangle_p
  = \sum_{\sigma\in \Gamma}\lambda_{\sigma}^{2p}\overline{\langle Z_{\sigma},\xi\rangle} \langle Z_{\sigma}, \eta\rangle,\quad
  \xi,\, \eta \in \mathcal{S}_p(Z).
\end{equation}
It is not hard to check that, with $\langle \cdot,\cdot\rangle_p$ as the inner product, $\mathcal{S}_p(Z)$ becomes a Hilbert space.
We write $\|\xi\|_{p}= \sqrt{\langle \xi,\xi\rangle_p}$ for $\xi \in \mathcal{S}_p(Z)$. Clearly, it holds that
\begin{equation}\label{eq-2-6}
  \|\xi\|_{p}^2=\sum_{\sigma\in \Gamma}\lambda_{\sigma}^{2p}|\langle Z_{\sigma}, \xi\rangle|^{2},\quad \xi \in \mathcal{S}_p(Z).
\end{equation}

\begin{lemma}\label{lem-2-3}\cite{wang-chen-1,wang-z}
For $p\geq 0$, one has $\{Z_{\sigma}\mid \sigma\in\Gamma\} \subset \mathcal{S}_p(Z)$ and moreover the system
$\{\lambda^{-p}_{\sigma}Z_{\sigma}\mid \sigma\in\Gamma\}$ forms an orthonormal basis for $\mathcal{S}_p(Z)$.
\end{lemma}

It is easy to see that $\lambda_{\sigma}\geq 1$ for all $\sigma\in \Gamma$. This implies that $\|\cdot\|_p \leq \|\cdot\|_q$
and $\mathcal{S}_q(Z)\subset \mathcal{S}_p(Z)$
whenever $0\leq p \leq q$. Thus we actually get a chain of Hilbert spaces of functionals of $Z$:
\begin{equation}\label{eq-2-7}
 \cdots \subset \mathcal{S}_{p+1}(Z) \subset \mathcal{S}_p(Z)\subset  \cdots \subset \mathcal{S}_1(Z) \subset \mathcal{S}_0(Z)=\mathcal{L}^2(Z).
\end{equation}
We now put
\begin{equation}\label{eq-2-8}
  \mathcal{S}(Z)=\bigcap ^{\infty}_{p=0}\mathcal{S}_{p}(Z)
\end{equation}
and endow it with the topology generated by the norm sequence $\{\|\cdot \|_{p}\}_{p\geq 0}$.
Note that, for each $ p\geq 0$, $\mathcal{S}_p(Z)$ is just the completion of $\mathcal{S}(Z)$ with respect to $\|\cdot\|_{p}$.
Thus $\mathcal{S}(Z)$ is a countably-Hilbert space \cite{becnel, gelfand}.
The next lemma, however, shows that $\mathcal{S}(Z)$ even has a much better property.

\begin{lemma}\label{lem-2-4}\cite{wang-chen-1,wang-z}
The space $\mathcal{S}(Z)$ is a nuclear space, namely for any $p\geq 0$,
there exists $q> p$ such that the inclusion mapping $i_{pq}\colon \mathcal{S}_{q}(Z) \rightarrow \mathcal{S}_p(Z)$
defined by $i_{pq}(\xi)=\xi$ is a Hilbert-Schmidt operator.
\end{lemma}

For $p\geq 0$, we denote by $\mathcal{S}_p^*(Z)$ the dual of $\mathcal{S}_p(Z)$ and $\|\cdot\|_{-p}$
the norm of $\mathcal{S}_p^*(Z)$. Then $\mathcal{S}_p^*(Z)\subset \mathcal{S}_q^*(Z)$ and
$\|\cdot\|_{-p} \geq \|\cdot\|_{-q}$ whenever $0\leq p \leq q$.
The lemma below is then an immediate consequence of the general theory of countably-Hilbert spaces (see, e.g., \cite{becnel} or \cite{gelfand}).

\begin{lemma}\label{lem-2-5}\cite{wang-chen-1,wang-z}
Let $\mathcal{S}^*(Z)$ the dual of $\mathcal{S}(Z)$ and endow it with the strong topology. Then
\begin{equation}\label{eq-2-9}
  \mathcal{S}^*(Z)=\bigcup_{p=0}^{\infty}\mathcal{S}_p^*(Z)
\end{equation}
and moreover the inductive limit topology over $\mathcal{S}^*(Z)$ given by space sequence $\{\mathcal{S}_p^*(Z)\}_{p\geq 0}$
coincides with the strong topology.
\end{lemma}

We mention that, by identifying $\mathcal{L}^2(Z)$ with its dual, one comes to a Gel'fand triple
\begin{equation}\label{eq-2-10}
\mathcal{S}(Z)\subset \mathcal{L}^2(Z)\subset \mathcal{S}^*(Z),
\end{equation}
which we refer to as the Gel'fand triple associated with the discrete-time normal noise $Z$.

\begin{theorem}\label{thr-2-6}\cite{wang-chen-1}
The system $\{Z_{\sigma} \mid \sigma \in \Gamma\}$ is contained in $\mathcal{S}(Z)$ and moreover it forms a basis for $\mathcal{S}(Z)$ in the sense that
\begin{equation}\label{eq-2-11}
  \xi = \sum_{\sigma \in \Gamma} \langle Z_{\sigma}, \xi\rangle Z_{\sigma}, \quad \xi \in \mathcal{S}(Z),
\end{equation}
where $\langle\cdot,\cdot\rangle$ is the inner product of $\mathcal{L}^2(Z)$ and the series converges in the topology of $\mathcal{S}(Z)$.
\end{theorem}

\begin{definition}\label{def-2-2}\cite{wang-chen-1,wang-z}
Elements of $\mathcal{S}^*(Z)$ are called generalized functionals of $Z$, while elements of $\mathcal{S}(Z)$ are called testing functionals of $Z$.
\end{definition}

Thus, $\mathcal{S}^*(Z)$ and $\mathcal{S}(Z)$ can be accordingly called the generalized functional space and the testing functional space of $Z$,
respectively. It turns out \cite{wang-chen-1} that $\mathcal{S}^*(Z)$ can accommodate
many quantities of theoretical interest that can not be covered by $\mathcal{L}^2(Z)$.

In the following, we denote by $\langle\!\langle \cdot,\cdot\rangle\!\rangle$
the canonical bilinear form on $\mathcal{S}^*(Z)\times \mathcal{S}(Z)$ given by
\begin{equation}\label{eq-2-12}
  \langle\!\langle \Phi,\xi\rangle\!\rangle = \Phi(\xi),\quad \Phi\in \mathcal{S}^*(Z),\, \xi\in \mathcal{S}(Z).
\end{equation}
Note that $\langle\!\langle \cdot,\cdot\rangle\!\rangle$ is different from the inner product
$\langle\cdot,\cdot\rangle$ of $\mathcal{L}^2(Z)$.

\begin{definition}\label{def-2-3}\cite{wang-chen-1}
For $\Phi \in \mathcal{S}^*(Z)$, its Fock transform is the function $\widehat{\Phi}$ on $\Gamma$ given by
\begin{equation}\label{eq-2-13}
  \widehat{\Phi}(\sigma) = \langle\!\langle \Phi, Z_{\sigma}\rangle\!\rangle,\quad \sigma \in \Gamma,
\end{equation}
where $\langle\!\langle \cdot,\cdot\rangle\!\rangle$ is the canonical bilinear form.
\end{definition}

It is easy to verify that,
for $\Phi$, $\Psi \in \mathcal{S}^*(Z)$, $\Phi=\Psi$ if and only if $\widehat{\Phi}=\widehat{\Psi}$.
Thus a generalized functional of $Z$ is completely determined by its Fock transform.
The following theorem characterizes generalized functionals of $Z$ through their Fock transforms.

\begin{theorem}\label{thr-2-7}\cite{wang-chen-1}
Let $F$ be a function on $\Gamma$. Then $F$ is the Fock transform of an element $\Phi$ of $\mathcal{S}^*(Z)$ if and only if it satisfies
\begin{equation}\label{eq-2-14}
  |F(\sigma)| \leq C\lambda_{\sigma}^p,\quad \sigma \in \Gamma
\end{equation}
for some constants $C\geq 0$ and $p\geq 0$.
In that case, for $q> p+\frac{1}{2}$, one has
\begin{equation}\label{eq-2-15}
  \|\Phi\|_{-q} \leq C\bigg[\sum_{\sigma \in \Gamma}\lambda_{\sigma}^{-2(q-p)}\bigg]^{\frac{1}{2}}
\end{equation}
and in particular $\Phi \in \mathcal{S}_q^*(Z)$.
\end{theorem}

The theorem below describes the regularity of generalized functionals of $Z$ via their Fock transforms.

\begin{theorem}\label{thr-2-8}
Let $\Phi \in \mathcal{S}^*(Z)$ and $p\geq 0$. Then $\Phi \in \mathcal{S}_p^*(Z)$ if and only if
\begin{equation}\label{eq-2-16}
  \sum_{\sigma \in \Gamma}\lambda_{\sigma}^{-2p}\big|\widehat{\Phi}(\sigma)\big|^2<\infty.
\end{equation}
In that case, the norm $\|\Phi\|_{-p}$ of $\Phi$ in $\mathcal{S}_p^*(Z)$ satisfies
\begin{equation}\label{eq-2-17}
  \|\Phi\|_{-p}^2= \sum_{\sigma \in \Gamma}\lambda_{\sigma}^{-2p}\big|\widehat{\Phi}(\sigma)\big|^2.
\end{equation}
\end{theorem}

\begin{proof}
The ``only if'' part. By the well-known Riesz representation theorem \cite{mus}, there exists a unique $\eta \in \mathcal{S}_p(Z)$
such that $\|\eta\|_p = \|\Phi\|_{-p}$ and
\begin{equation*}
  \Phi(\xi) = \langle \eta, \xi\rangle_p,\quad \xi \in \mathcal{S}_p(Z).
\end{equation*}
Thus
\begin{equation*}
   \sum_{\sigma \in \Gamma}\lambda_{\sigma}^{-2p}\big|\widehat{\Phi}(\sigma)\big|^2
   =\sum_{\sigma \in \Gamma}\lambda_{\sigma}^{-2p}|\langle Z_{\sigma}, \eta\rangle_p|^2
   =\sum_{\sigma \in \Gamma}\lambda_{\sigma}^{2p}|\langle Z_{\sigma}, \eta\rangle|^2
   =\|\eta\|_p^2
   =\|\Phi\|_{-p}^2,
\end{equation*}
which implies (\ref{eq-2-16}) and (\ref{eq-2-17}).

The ``if'' part. For each $\xi \in \mathcal{S}(Z)$, by using Theorem~\ref{thr-2-6}, we have
\begin{equation*}
\begin{split}
  |\Phi(\xi)|
  =\Big|\sum_{\sigma \in \Gamma}\langle Z_{\sigma}, \xi\rangle\Phi(Z_{\sigma})\Big|
  &= \Big|\sum_{\sigma \in \Gamma}\langle Z_{\sigma}, \xi\rangle\widehat{\Phi}(\sigma)\Big|\\
  &\leq \Big[\sum_{\sigma \in \Gamma} \lambda_{\sigma}^{2p}|\langle Z_{\sigma}, \xi\rangle|^2\Big]^{\frac{1}{2}}
        \Big[\sum_{\sigma \in \Gamma} \lambda_{\sigma}^{-2p}|\widehat{\Phi}(\sigma)|^2\Big]^{\frac{1}{2}}\\
  & = \|\xi\|_p\Big[\sum_{\sigma \in \Gamma} \lambda_{\sigma}^{-2p}|\widehat{\Phi}(\sigma)|^2\Big]^{\frac{1}{2}}.
\end{split}
\end{equation*}
Thus $\Phi$ is a bounded functional on the space $(\mathcal{S}(Z),\|\cdot\|_p)$, which implies $\Phi \in \mathcal{S}_p^*(Z)$
since $\mathcal{S}(Z)$ is dense in $\mathcal{S}_p(Z)$.
\end{proof}

\begin{remark}\label{rem-2-1}
There exists a continuous linear mapping $\textit{\textsf{R}}\colon \mathcal{L}^2(Z) \rightarrow \mathcal{S}^*(Z)$ such that
\begin{equation}\label{eq-2-18}
  \langle\!\langle\textit{\textsf{R}}\eta, \xi \rangle\!\rangle
= \langle \eta, \xi\rangle,\quad \eta \in \mathcal{L}^2(Z),\, \xi \in \mathcal{S}(Z),
\end{equation}
where $\langle\!\langle\cdot, \cdot \rangle\!\rangle$ is the canonical bilinear form on $\mathcal{S}^*(Z)\times \mathcal{S}(Z)$.
We call $\textit{\textsf{R}}$ the Riesz mapping.
\end{remark}

\begin{theorem}\label{thr-2-9}\cite{wang-chen-2}
Let $\Phi$, $\Phi_n \in \mathcal{S}^*(Z)$, $n\geq 1$, be generalized functionals of $Z$.  Then the sequence $(\Phi_n)$ converges strongly
to $\Phi$ in $\mathcal{S}^*(Z)$ if and only if it satisfies:
\begin{enumerate}
  \item[(1)] $\widehat{\Phi_n}(\sigma)\rightarrow \widehat{\Phi}(\sigma)$ for all $\sigma\in \Gamma$;
  \item[(2)] There are constants $C\geq 0$ and $p\geq 0$ such that
\begin{equation}\label{eq-2-19}
   \sup_{n\geq 1}\big|\widehat{\Phi_n}(\sigma)\big| \leq C\lambda_{\sigma}^p,\quad \sigma \in \Gamma.
\end{equation}
\end{enumerate}
\end{theorem}

\section{Clark-Ocone formula for generalized functionals}\label{sec-3}

In this section, we first introduce some fundamental operators on the space
$\mathcal{S}^*(Z)$. And then we establish our Clark-Ocone for functionals in $\mathcal{S}^*(Z)$.

\subsection{Annihilation and creation operators}\label{subsec-3-1}

\begin{theorem}\label{thr-3-1}
 Let $k\in\mathbb{N}$. Then there exists a continuous linear operator $\mathfrak{a}_k\colon \mathcal{S}^*(Z) \rightarrow \mathcal{S}^*(Z)$
 such that
\begin{equation}\label{eq-annihilation}
\widehat{\mathfrak{a}_k\Phi}(\sigma)=[1-\mathbf{1}_{\sigma}(k)]\widehat{\Phi}(\sigma\cup k),\quad
 \sigma\in\Gamma,\,\Phi\in\mathcal{S}^*(Z).
\end{equation}
\end{theorem}

\begin{proof}
For each $\Phi\in\mathcal{S}^*(Z)$, by Theorem~\ref{thr-2-7}, there exist constants $C$, $p\geq 0$ such that
\begin{equation*}
  |\widehat{\Phi}(\sigma)|\leq C\lambda_{\sigma}^p,\quad  \sigma \in \Gamma,
\end{equation*}
which means that the function $\sigma \mapsto [1-\mathbf{1}_{\sigma}(k)]\widehat{\Phi}(\sigma\cup k)$ satisfies
\begin{equation*}
\begin{split}
  \big|[1-\mathbf{1}_{\sigma}(k)]\widehat{\Phi}(\sigma\cup k)\big|
   &\leq  [1-\mathbf{1}_{\sigma}(k)]C\lambda_{\sigma\cup k}^p\\
  &=[1-\mathbf{1}_{\sigma}(k)]C(1+k)^p\lambda_{\sigma}^p
  \leq C(1+k)^p\lambda_{\sigma}^p,\quad  \sigma\in\Gamma,
\end{split}
\end{equation*}
which, together with Theorem~\ref{thr-2-7}, implies that there exists a unique $\Psi_{\Phi}\in \mathcal{S}^*(Z)$ such that
\begin{equation}\label{eq-3-2}
  \widehat{\Psi_{\Phi}}(\sigma) = [1-\mathbf{1}_{\sigma}(k)]\widehat{\Phi}(\sigma\cup k),\quad \sigma\in \Gamma.
\end{equation}
Now, consider the mapping $\mathfrak{a}_k\colon \mathcal{S}^*(Z) \rightarrow \mathcal{S}^*(Z)$ defined by
\begin{equation}\label{eq-3-3}
  \mathfrak{a}_k\Phi = \Psi_{\Phi},\quad \Phi\in \mathcal{S}^*(Z).
\end{equation}
It is not hard to verify that $\mathfrak{a}_k$ is a linear operator and satisfies (\ref{eq-annihilation}).
To complete the proof, we still need to show that $\mathfrak{a}_k\colon \mathcal{S}^*(Z) \rightarrow \mathcal{S}^*(Z)$ is continuous
with respect to the strong topology over $\mathcal{S}^*(Z)$.

Let $p\geq 0$ and denote by $\mathfrak{j}_k \colon \mathcal{S}_p^*(Z)\rightarrow \mathcal{S}^*(Z)$ the inclusion mapping,
namely $\mathfrak{j}_k$ is the mapping defined by
\begin{equation}\label{eq-3-4}
  \mathfrak{j}_k(\Phi) = \Phi,\quad \Phi\in \mathcal{S}_p^*(Z).
\end{equation}
Then the composition mapping $\mathfrak{a}_k\circ \mathfrak{j}_k$ is a linear operator from $\mathcal{S}_p^*(Z)$ to $\mathcal{S}^*(Z)$.
For each $\Phi \in \mathcal{S}_p^*(Z)$, we have
\begin{equation*}
\begin{split}
  \sum_{\sigma \in \Gamma}\lambda_{\sigma}^{-2p}\big|\widehat{\mathfrak{a}_k\circ \mathfrak{j}_k(\Phi)}(\sigma)\big|^2
  &= \sum_{\sigma \in \Gamma}\lambda_{\sigma}^{-2p}\big|\widehat{\mathfrak{a}_k\Phi}(\sigma)\big|^2\\
  &= \sum_{\sigma \in \Gamma}\lambda_{\sigma}^{-2p}\big|[1-\mathbf{1}_{\sigma}(k)]\widehat{\Phi}(\sigma\cup k)\big|^2\\
  &= \sum_{k\notin\sigma \in \Gamma}(1+k)^{2p}\lambda_{\sigma\cup k}^{-2p}\big|\widehat{\Phi}(\sigma\cup k)\big|^2\\
  &\leq (1+k)^{2p}\sum_{\tau \in \Gamma}\lambda_{\tau}^{-2p}\big|\widehat{\Phi}(\tau)\big|^2,
\end{split}
\end{equation*}
which together with Theorem~\ref{thr-2-8} implies that $\mathfrak{a}_k\circ \mathfrak{j}_k(\Phi)\in \mathcal{S}_p^*(Z)$
and
\begin{equation*}
  \big\|\mathfrak{a}_k\circ \mathfrak{j}_k(\Phi)\big\|_{-p} \leq (1+k)^p  \big\|\Phi\big\|_{-p}.
\end{equation*}
Thus $\mathfrak{a}_k\circ \mathfrak{j}_k\big(\mathcal{S}_p^*(Z)\big)\subset \mathcal{S}_p^*(Z)$
and $\mathfrak{a}_k\circ\mathfrak{j}_k \colon \mathcal{S}_p^*(Z)\rightarrow \mathcal{S}_p^*(Z)$
is a bounded operator,
which implies that $\mathfrak{a}_k\circ \mathfrak{j}_k$ is continuous as an operator
from $\mathcal{S}_p^*(Z)$ to $\mathcal{S}^*(Z)$.

Since the choice of the above $p\geq 0$ is arbitrary,  we actually arrive at a conclusion that
the composition mapping $\mathfrak{a}_k\circ \mathfrak{j}_k\colon \mathcal{S}_p^*(Z)\rightarrow \mathcal{S}^*(Z)$
is continuous for all $p\geq 0$.
Therefore $\mathfrak{a}_k\colon \mathcal{S}^*(Z) \rightarrow \mathcal{S}^*(Z)$
is continuous with respect to the inductive limit topology over $\mathcal{S}^*(Z)$,
which together with Lemma~\ref{lem-2-5} implies that $\mathfrak{a}_k\colon \mathcal{S}^*(Z) \rightarrow \mathcal{S}^*(Z)$
is continuous with respect to the strong topology over $\mathcal{S}^*(Z)$.
\end{proof}

Carefully checking the proof of Theorem~\ref{thr-3-1}, one can find the next result already proven.

\begin{theorem}\label{thr-3-2}
Let $k\in\mathbb{N}$. Then, for each $p\geq 0$, $\mathcal{S}_p^*(Z)$ keeps invariant under the action of $\mathfrak{a}_k$,
and moreover
\begin{equation}\label{eq-3-5}
  \big\|\mathfrak{a}_k\Phi\big\|_{-p} \leq (1+k)^p  \big\|\Phi\big\|_{-p},\quad \Phi \in \mathcal{S}_p^*(Z).
\end{equation}
\end{theorem}

With the same arguments, we can prove the next two theorems,
which are dual forms of Theorem~\ref{thr-3-1} and Theorem~\ref{thr-3-2}, respectively.

\begin{theorem}\label{thr-3-3}
 Let $k\in\mathbb{N}$. Then there exists a continuous linear operator $\mathfrak{a}_k^{\dag}\colon \mathcal{S}^*(Z) \rightarrow \mathcal{S}^*(Z)$
 such that
\begin{equation}\label{eq-creation}
\widehat{\mathfrak{a}_k^{\dag}\Phi}(\sigma)=\mathbf{1}_{\sigma}(k)\widehat{\Phi}(\sigma\setminus k),\quad
 \sigma\in\Gamma,\,\Phi\in\mathcal{S}^*(Z).
\end{equation}
\end{theorem}

\begin{proof}
For each $\Phi\in\mathcal{S}^*(Z)$, by Theorem~\ref{thr-2-7}, there exist constants $C$, $p\geq 0$ such that
\begin{equation*}
  |\widehat{\Phi}(\sigma)|\leq C\lambda_{\sigma}^p,\quad  \sigma \in \Gamma,
\end{equation*}
which means that the function $\sigma \mapsto \mathbf{1}_{\sigma}(k)\widehat{\Phi}(\sigma\setminus k)$ satisfies
\begin{equation*}
\begin{split}
  \big|\mathbf{1}_{\sigma}(k)\widehat{\Phi}(\sigma\setminus k)\big|
   &\leq  \mathbf{1}_{\sigma}(k)C\lambda_{\sigma\setminus k}^p\\
  &=\mathbf{1}_{\sigma}(k)C(1+k)^{-p}\lambda_{\sigma}^p
  \leq C(1+k)^{-p}\lambda_{\sigma}^p,\quad  \sigma\in\Gamma,
\end{split}
\end{equation*}
which, together with Theorem~\ref{thr-2-7}, implies that there exists a unique $\Theta_{\Phi}\in \mathcal{S}^*(Z)$ such that
\begin{equation}\label{eq-3-7}
  \widehat{\Theta_{\Phi}}(\sigma) = \mathbf{1}_{\sigma}(k)\widehat{\Phi}(\sigma\setminus k),\quad \sigma\in \Gamma.
\end{equation}
Now, consider the mapping $\mathfrak{a}_k^{\dag}\colon \mathcal{S}^*(Z) \rightarrow \mathcal{S}^*(Z)$ defined by
\begin{equation}\label{eq-3-8}
  \mathfrak{a}_k^{\dag}\Phi = \Theta_{\Phi},\quad \Phi\in \mathcal{S}^*(Z).
\end{equation}
It is not hard to verify that $\mathfrak{a}_k^{\dag}$ is a linear operator and satisfies (\ref{eq-creation}).
To complete the proof, we still need to show that $\mathfrak{a}_k^{\dag}\colon \mathcal{S}^*(Z) \rightarrow \mathcal{S}^*(Z)$ is continuous
with respect to the strong topology over $\mathcal{S}^*(Z)$.

Let $p\geq 0$ and denote by $\mathfrak{j}_k \colon \mathcal{S}_p^*(Z)\rightarrow \mathcal{S}^*(Z)$ the inclusion mapping.
Then the composition mapping $\mathfrak{a}_k^{\dag}\circ \mathfrak{j}_k$ is a linear operator from $\mathcal{S}_p^*(Z)$ to $\mathcal{S}^*(Z)$.
For each $\Phi \in \mathcal{S}_p^*(Z)$, we have
\begin{equation*}
\begin{split}
  \sum_{\sigma \in \Gamma}\lambda_{\sigma}^{-2p}\big|\widehat{\mathfrak{a}_k^{\dag}\circ \mathfrak{j}_k(\Phi)}(\sigma)\big|^2
  &= \sum_{\sigma \in \Gamma}\lambda_{\sigma}^{-2p}\big|\widehat{\mathfrak{a}_k^{\dag}\Phi}(\sigma)\big|^2\\
  &= \sum_{\sigma \in \Gamma}\lambda_{\sigma}^{-2p}\big|\mathbf{1}_{\sigma}(k)\widehat{\Phi}(\sigma\setminus k)\big|^2\\
  &= \sum_{k\in\sigma \in \Gamma}(1+k)^{-2p}\lambda_{\sigma\setminus k}^{-2p}\big|\widehat{\Phi}(\sigma\setminus k)\big|^2\\
  &\leq (1+k)^{-2p}\sum_{\tau \in \Gamma}\lambda_{\tau}^{-2p}\big|\widehat{\Phi}(\tau)\big|^2,
\end{split}
\end{equation*}
which together with Theorem~\ref{thr-2-8} implies that $\mathfrak{a}_k^{\dag}\circ \mathfrak{j}_k(\Phi)\in \mathcal{S}_p^*(Z)$
and
\begin{equation*}
  \big\|\mathfrak{a}_k^{\dag}\circ \mathfrak{j}_k(\Phi)\big\|_{-p} \leq (1+k)^{-p}  \big\|\Phi\big\|_{-p}.
\end{equation*}
Thus $\mathfrak{a}_k^{\dag}\circ \mathfrak{j}_k\big(\mathcal{S}_p^*(Z)\big)\subset \mathcal{S}_p^*(Z)$
and $\mathfrak{a}_k^{\dag}\circ\mathfrak{j}_k \colon \mathcal{S}_p^*(Z)\rightarrow \mathcal{S}_p^*(Z)$
is a bounded operator,
which implies that $\mathfrak{a}_k^{\dag}\circ \mathfrak{j}_k$ is continuous as an operator
from $\mathcal{S}_p^*(Z)$ to $\mathcal{S}^*(Z)$.

Since the choice of the above $p\geq 0$ is arbitrary,  we actually arrive at a conclusion that
the composition mapping $\mathfrak{a}_k^{\dag}\circ \mathfrak{j}_k\colon \mathcal{S}_p^*(Z)\rightarrow \mathcal{S}^*(Z)$
is continuous for all $p\geq 0$.
Therefore $\mathfrak{a}_k^{\dag}\colon \mathcal{S}^*(Z) \rightarrow \mathcal{S}^*(Z)$
is continuous with respect to the inductive limit topology over $\mathcal{S}^*(Z)$,
which together with Lemma~\ref{lem-2-5} implies that $\mathfrak{a}_k^{\dag}\colon \mathcal{S}^*(Z) \rightarrow \mathcal{S}^*(Z)$
is continuous with respect to the strong topology over $\mathcal{S}^*(Z)$.
\end{proof}

From the proof of Theorem~\ref{thr-3-3}, we can easily get the next result concerning the operator $\mathfrak{a}_k^{\dag}$.

\begin{theorem}\label{thr-3-4}
Let $k\in\mathbb{N}$. Then, for each $p\geq 0$, $\mathcal{S}_p^*(Z)$ keeps invariant under the action of $\mathfrak{a}_k^{\dag}$,
and moreover
\begin{equation}\label{eq-3-9}
  \big\|\mathfrak{a}_k^{\dag}\Phi\big\|_{-p} \leq (1+k)^{-p}  \big\|\Phi\big\|_{-p},\quad \Phi \in \mathcal{S}_p^*(Z).
\end{equation}
\end{theorem}

\begin{remark}\label{rem-3-1}
For $k\geq 0$, the corresponding annihilation operator $\partial_k$ on $\mathcal{L}^2(Z)$ and its dual $\partial_k^*$ (known as the creation operator)
admit the property
\begin{equation*}
  \partial_k Z_{\sigma} = \mathbf{1}_{\sigma}(k)Z_{\sigma\setminus k},\quad
  \partial_k^* Z_{\sigma} = [1-\mathbf{1}_{\sigma}(k)]Z_{\sigma\cup k}, \quad \sigma \in \Gamma.
\end{equation*}
And moreover, they satisfy the canonical anti-commutation relation (CAR) in equal-time
\begin{equation*}
\partial_k^*\partial_k + \partial_k\partial_k^*=I,
\end{equation*}
where $I$ means the identity operator on $\mathcal{L}^2(Z)$. We refer to \cite{wang-lc, wang-cl} for details about these operators.
\end{remark}

The next theorem shows the link between $\mathfrak{a}_k$ and $\partial_k$, as well as between $\mathfrak{a}_k^{\dag}$
and $\partial_k^*$.

\begin{theorem}\label{thr-3-5}
Let $k\geq 0$. Then the operators $\mathfrak{a}_k$ and $\mathfrak{a}_k^{\dag}$ satisfy
\begin{equation}\label{eq-3-10}
  \mathfrak{a}_k \textsf{R} = \textsf{R}\partial_k,\quad \mathfrak{a}_k^{\dag}\textsf{R} = \textsf{R}\partial_k^*.
\end{equation}
where $\textit{\textsf{R}}$ is the Riesz mapping as indicated in Remark~\ref{rem-2-1}.
\end{theorem}

\begin{proof}
Let $\eta \in \mathcal{L}^2(Z)$. Then, for all $\sigma \in \Gamma$, we have
\begin{equation*}
  \widehat{\mathfrak{a}_k \textit{\textsf{R}}\eta}(\sigma)
  = [1-\mathbf{1}_{\sigma}(k)]\langle \eta, Z_{\sigma\cup k}\rangle
  = \langle \eta, \partial_k^* Z_{\sigma}\rangle
  = \langle \partial_k\eta,  Z_{\sigma}\rangle
  = \widehat{\textit{\textsf{R}}\partial_k \eta}(\sigma),
\end{equation*}
which implies $\mathfrak{a}_k \textit{\textsf{R}}\eta= \textit{\textsf{R}}\partial_k \eta$. It then follows by
the arbitrariness of $\eta \in \mathcal{L}^2(Z)$ that $\mathfrak{a}_k \textit{\textsf{R}}= \textit{\textsf{R}}\partial_k$.
Similarly, we can prove $\mathfrak{a}_k^{\dag}\textit{\textsf{R}} = \textit{\textsf{R}}\partial_k^*$.
\end{proof}

In view of Theorem~\ref{thr-3-5}, we give the following definition to name the operators $\mathfrak{a}_k$ and $\mathfrak{a}_k^{\dag}$.

\begin{definition}\label{def-3-1}
For $k\geq 0$, the operators $\mathfrak{a}_k$ and $\mathfrak{a}_k^{\dag}$ are called the
annihilation and creation operators on generalized functionals of $Z$, respectively.
\end{definition}

Much like the operators $\{\partial_k,\,\partial_k^*\}$ on $\mathcal{L}^2(Z)$,
the operators $\{\mathfrak{a}_k,\, \mathfrak{a}_k^{\dag}\}$ also satisfy a canonical anti-commutation relation (CAR) in equal-time.

\begin{theorem}\label{thr-3-6}
Let $I$ be the identity operator on $\mathcal{S}^*(Z)$. Then, for $k\geq 0$, it holds that
\begin{equation}\label{eq-3-11}
\mathfrak{a}_k^{\dag}\mathfrak{a}_k + \mathfrak{a}_k\mathfrak{a}_k^{\dag} =I.
\end{equation}
\end{theorem}

\begin{proof}
Let $\Phi \in \mathcal{S}^*(Z)$. Then, for any $\sigma \in \Gamma$, it follows from (\ref{eq-annihilation}) and (\ref{eq-creation}) that
\begin{equation}\label{eq-3-12}
  \widehat{\mathfrak{a}_k^{\dag}\mathfrak{a}_k\Phi}(\sigma)
  = \mathbf{1}_{\sigma}(k)\widehat{\mathfrak{a}_k\Phi}(\sigma\setminus k)
  = \mathbf{1}_{\sigma}(k) \widehat{\Phi}(\sigma)
\end{equation}
and
\begin{equation}\label{eq-3-13}
  \widehat{\mathfrak{a}_k\mathfrak{a}_k^{\dag}\Phi}(\sigma)
  = (1-\mathbf{1}_{\sigma}(k))\widehat{\mathfrak{a}_k^{\dag}\Phi}(\sigma \cup k)
  = (1-\mathbf{1}_{\sigma}(k))\widehat{\Phi}(\sigma),
\end{equation}
thus
\begin{equation*}
  \widehat{(\mathfrak{a}_k^{\dag}\mathfrak{a}_k + \mathfrak{a}_k\mathfrak{a}_k^{\dag})\Phi}(\sigma)
  = \widehat{\mathfrak{a}_k\mathfrak{a}_k^{\dag}\Phi}(\sigma) + \widehat{\mathfrak{a}_k\mathfrak{a}_k^{\dag}\Phi}(\sigma)
  = \widehat{\Phi}(\sigma),
\end{equation*}
which implies that $(\mathfrak{a}_k^{\dag}\mathfrak{a}_k + \mathfrak{a}_k\mathfrak{a}_k^{\dag})\Phi=\Phi$.
It then follows from the arbitrariness of $\Phi \in \mathcal{S}^*(Z)$
that $\mathfrak{a}_k^{\dag}\mathfrak{a}_k + \mathfrak{a}_k\mathfrak{a}_k^{\dag} =I$.
\end{proof}

\subsection{Expectation and conditional expectation operators}\label{subsec-3-2}

For the Riesz mapping $\textit{\textsf{R}}$, by using Theorem~\ref{thr-2-8}, we can prove that
$\textit{\textsf{R}}\eta \in \mathcal{S}_0^*(Z)$ for all $\eta \in \mathcal{L}^2(Z)$.
Especially, we have $\textit{\textsf{R}}1 \in \mathcal{S}_0^*(Z)$.

\begin{theorem}\label{thr-3-7}
The mapping $\mathfrak{E}\colon \mathcal{S}^*(Z) \rightarrow \mathcal{S}^*(Z)$ defined by
\begin{equation}\label{eq-3-14}
  \mathfrak{E}\Phi = \widehat{\Phi}(\emptyset)\textit{\textsf{R}}1,\quad \Phi \in \mathcal{S}^*(Z)
\end{equation}
is a continuous linear operator from $\mathcal{S}^*(Z)$ to itself. And moreover,
 \begin{equation}\label{eq-3-15}
   \widehat{\mathfrak{E}\Phi}(\sigma) = \widehat{\Phi}(\emptyset)\langle 1, Z_{\sigma}\rangle,\quad \sigma\in \Gamma,\, \Phi\in \mathcal{S}^*(Z).
 \end{equation}
\end{theorem}

\begin{proof}
Clearly, $\mathfrak{E}\colon \mathcal{S}^*(Z) \rightarrow \mathcal{S}^*(Z)$ is a linear operator
and satisfies (\ref{eq-3-15}).
Next, let us show that $\mathfrak{E}\colon \mathcal{S}^*(Z) \rightarrow \mathcal{S}^*(Z)$ is continuous
with respect to the strong topology over $\mathcal{S}^*(Z)$.

Let $p\geq 0$ and denote by $\mathfrak{j}_k \colon \mathcal{S}_p^*(Z)\rightarrow \mathcal{S}^*(Z)$ the inclusion mapping.
Then the composition mapping $\mathfrak{E}\circ \mathfrak{j}_k$ is a linear operator from $\mathcal{S}_p^*(Z)$ to $\mathcal{S}^*(Z)$.
For each $\Phi \in \mathcal{S}_p^*(Z)$, we have
\begin{equation*}
\begin{split}
  \sum_{\sigma \in \Gamma}\lambda_{\sigma}^{-2p}\big|\widehat{\mathfrak{E}\circ \mathfrak{j}_k(\Phi)}(\sigma)\big|^2
  &= \sum_{\sigma \in \Gamma}\lambda_{\sigma}^{-2p}\big|\widehat{\mathfrak{E}\Phi}(\sigma)\big|^2\\
  &= \sum_{\sigma \in \Gamma}\lambda_{\sigma}^{-2p}\big|\widehat{\Phi}(\emptyset)\langle 1, Z_{\sigma}\rangle \big|^2\\
  &\leq \sum_{\sigma \in \Gamma}\lambda_{\sigma}^{-2p}\big|\widehat{\Phi}(\sigma)\big|^2,
\end{split}
\end{equation*}
which together with Theorem~\ref{thr-2-8} implies that $\mathfrak{E}\circ \mathfrak{j}_k(\Phi)\in \mathcal{S}_p^*(Z)$
and
\begin{equation*}
  \big\|\mathfrak{E}\circ \mathfrak{j}_k(\Phi)\big\|_{-p} \leq  \big\|\Phi\big\|_{-p}.
\end{equation*}
Thus $\mathfrak{E}\circ \mathfrak{j}_k\big(\mathcal{S}_p^*(Z)\big)\subset \mathcal{S}_p^*(Z)$
and $\mathfrak{E}\circ\mathfrak{j}_k \colon \mathcal{S}_p^*(Z)\rightarrow \mathcal{S}_p^*(Z)$
is a bounded operator,
which implies that $\mathfrak{E}\circ \mathfrak{j}_k$ is continuous as an operator
from $\mathcal{S}_p^*(Z)$ to $\mathcal{S}^*(Z)$.

Since the choice of the above $p\geq 0$ is arbitrary,  we actually arrive at a conclusion that
the composition mapping $\mathfrak{E}\circ \mathfrak{j}_k\colon \mathcal{S}_p^*(Z)\rightarrow \mathcal{S}^*(Z)$
is continuous for all $p\geq 0$.
Therefore $\mathfrak{E}\colon \mathcal{S}^*(Z) \rightarrow \mathcal{S}^*(Z)$
is continuous with respect to the inductive limit topology over $\mathcal{S}^*(Z)$,
which together with Lemma~\ref{lem-2-5} implies that $\mathfrak{E}\colon \mathcal{S}^*(Z) \rightarrow \mathcal{S}^*(Z)$
is continuous with respect to the strong topology over $\mathcal{S}^*(Z)$.
\end{proof}

\begin{definition}\label{def-3-2}
The operator $\mathfrak{E}$ is called the expectation operator on generalized functionals of $Z$.
\end{definition}

Since $1\in \mathcal{L}^2(Z)$, the expectation $\mathbb{E}$ with respect to $P$ is actually a bounded operator from $\mathcal{L}^2(Z)$ to itself.
The next theorem shows the link between the operators $\mathfrak{E}$ and $\mathbb{E}$, which justifies the above definition.

\begin{theorem}\label{thr-3-8}
It holds that $\mathfrak{E}\textit{\textsf{R}} = \textit{\textsf{R}}\,\mathbb{E}$, where $\textit{\textsf{R}}$ is the Riesz mapping.
\end{theorem}

\begin{proof}
For any $\xi \in \mathcal{L}^2(Z)$ and any $\sigma\in \Gamma$, by a direct computation, we have
\begin{equation*}
  \widehat{\textit{\textsf{R}}\,\mathbb{E}\xi}(\sigma)
= \langle \mathbb{E}\xi, Z_{\sigma}\rangle
= \langle \xi, Z_{\emptyset}\rangle \langle 1, Z_{\sigma}\rangle
= \widehat{\textit{\textsf{R}}\xi}(\emptyset)\langle 1, Z_{\sigma}\rangle
= \widehat{\mathfrak{E}\textit{\textsf{R}}\xi}(\sigma).
\end{equation*}
Thus $\mathfrak{E}\textit{\textsf{R}} = \textit{\textsf{R}}\,\mathbb{E}$.
\end{proof}

\begin{theorem}\label{thr-3-9}
Let $k\geq 0$. Then there exists a continuous linear operator $\mathfrak{E}_k\colon \mathcal{S}^*(Z) \rightarrow \mathcal{S}^*(Z)$
 such that
 \begin{equation}\label{eq-conditional}
   \widehat{\mathfrak{E}_k\Phi}(\sigma) = \mathbf{1}_{\Gamma\!_{k]}}(\sigma)\widehat{\Phi}(\sigma),\quad \sigma\in \Gamma,
 \end{equation}
where $\Gamma\!_{k]} = \{\sigma \in \Gamma \mid \max \sigma \leq k\}$ and $\mathbf{1}_{\Gamma\!_{k]}}(\cdot)$ denotes the indicator of $\Gamma\!_{k]}$.
\end{theorem}

\begin{proof}
We omit the proof because it is quite similar to that of Theorem~\ref{thr-3-1}.
\end{proof}

By using Theorem~\ref{thr-2-8} and Theorem~\ref{thr-3-9}, we can easily prove the next theorem,
which shows that the operator $\mathfrak{E}_k$ has a type of contraction property on $\mathcal{S}^*(Z)$.

\begin{theorem}\label{thr-3-10}
Let $k\geq 0$. Then, for each $p\geq 0$, $\mathcal{S}_p^*(Z)$ keeps invariant under the action of $\mathfrak{E}_k$,
and moreover
\begin{equation}\label{}
  \|\mathfrak{E}_k\Phi\|_{-p} \leq \|\Phi\|_{-p},\quad \forall\, \Phi \in \mathcal{S}_p^*(Z).
\end{equation}
\end{theorem}

\begin{definition}\label{def-3-3}
The operators $\mathfrak{E}_k$, $k\geq 0$,  are called the conditional expectation operators
on generalized functionals of $Z$.
\end{definition}

For $k\geq 0$, we set $P_k = \mathbb{E}[\cdot\,|\, \mathcal{F}_k]$,
the expectation given $\mathcal{F}_k$, where $\mathcal{F}_k$ is the $\sigma$-filed generated by $(Z_j; 0\leq j \leq k)$ as mentioned above.
$P_k$ is usually known as a conditional expectation operator on square integrable functionals of $Z$.
The theorem below then justifies Definition~\ref{def-3-3}.

\begin{theorem}\label{thr-3-11}
For each $k\geq 0$, it holds that
$\mathfrak{E}_k\textit{\textsf{R}}= \textit{\textsf{R}} P_k$, where $\textit{\textsf{R}}$ is the Riesz mapping.
\end{theorem}

\begin{proof}
Let $k\geq 0$. Then,
for any $\xi \in \mathcal{L}^2(Z)$ and any $\sigma\in \Gamma$, by a direct computation, we have
\begin{equation*}
  \widehat{\textit{\textsf{R}}P_k\xi}(\sigma)
= \langle P_k\xi, Z_{\sigma}\rangle
= \langle \xi, P_kZ_{\sigma}\rangle
= \mathbf{1}_{\Gamma\!_{k]}}(\sigma)\langle \xi, Z_{\sigma}\rangle
= \mathbf{1}_{\Gamma\!_{k]}}(\sigma) \widehat{\textit{\textsf{R}}\xi}(\sigma)
= \widehat{\mathfrak{E}_k\textit{\textsf{R}}\xi}(\sigma).
\end{equation*}
Thus $\mathfrak{E}_k\textit{\textsf{R}} = \textit{\textsf{R}}P_k$.
\end{proof}

\subsection{Clark-Ocone formula for generalized functionals}\label{subsec-3-3}

In this subsection, we establish our Clark-Ocone formula for generalized functionals of $Z$.

\begin{theorem}\label{thr-3-12}
For all generalized functional $\Phi\in \mathcal{S}^*(Z)$, it holds that
 \begin{equation}\label{eq-3-18}
   \Phi = \mathfrak{E}\Phi + \sum_{k=0}^{\infty}\mathfrak{E}_k\mathfrak{a}_k^{\dag}\mathfrak{a}_k\Phi,
 \end{equation}
where the series on the righthand side converges strongly in $\mathcal{S}^*(Z)$.
\end{theorem}

\begin{proof}
Let $\Phi \in \mathcal{S}^*(Z)$ and $\Psi_n= \sum_{k=0}^n\mathfrak{E}_k\mathfrak{a}_k^{\dag}\mathfrak{a}_k\Phi$ for $n\geq 0$.
Then, for $\sigma \in \Gamma$, by a direct computation,
we have
\begin{equation}\label{eq-3-19}
  \widehat{\Psi_n}(\sigma)
= \sum_{k=0}^n\mathbf{1}_{\Gamma\!_{k]}}(\sigma)\mathbf{1}_{\sigma}(k)\widehat{\Phi}(\sigma)
=
\left\{
  \begin{array}{ll}
    0, & \hbox{$\sigma = \emptyset$;} \\
    0, & \hbox{$\sigma \neq \emptyset$, $n< \max \sigma$;} \\
    \widehat{\Phi}(\sigma), & \hbox{$\sigma \neq \emptyset$, $n\geq \max \sigma$.}
  \end{array}
\right.
\end{equation}
It then follows that $\widehat{\Psi_n}(\sigma) \rightarrow \widehat{\Phi - \mathfrak{E}\Phi}(\sigma)$ for all $\sigma \in \Gamma$
as $n\rightarrow \infty$.
On the other hand, by Theorem~\ref{thr-2-7}, there are constants $C\geq 0$ and $p\geq 0$ such that
\begin{equation*}
  \big|\widehat{\Phi}(\sigma)\big| \leq C \lambda_{\sigma}^p,\quad \sigma\in \Gamma,
\end{equation*}
which together with (\ref{eq-3-20}) gives
\begin{equation*}
  \sup_{n\geq 0}\big|\widehat{\Psi_n}(\sigma)\big|\leq \big|\widehat{\Phi}(\sigma)\big|\leq C \lambda_{\sigma}^p,\quad \sigma\in \Gamma.
\end{equation*}
Therefore, by Theorem~\ref{thr-2-9}, we know $(\Psi_n)$ converges strongly to $\Phi - \mathfrak{E}\Phi$ in $\mathcal{S}^*(Z)$.
This completes the proof.
\end{proof}

\begin{proposition}\label{prop-3-13}
For each $k\geq 0$, it holds that
\begin{equation}\label{eq-3-20}
\mathfrak{E}_k\mathfrak{a}_k^{\dag} = \mathfrak{a}_k^{\dag}\mathfrak{E}_k,\quad
\mathfrak{E}_k\mathfrak{a}_k = \mathfrak{E}_{k-1}\mathfrak{a}_k,
\end{equation}
where $\mathfrak{E}_{-1} = \mathfrak{E}$.
\end{proposition}

\begin{proof}
Let $k\geq 0$. Then,
for all $\Phi\in \mathcal{S}^*(Z)$ and $\sigma \in \Gamma$, by Theorems~\ref{thr-3-3} and ~\ref{thr-3-9}, we get
\begin{equation*}
 \widehat{\mathfrak{E}_k\mathfrak{a}_k^{\dag}\Phi}(\sigma)
= \mathbf{1}_{\Gamma\!_{k]}}(\sigma)\mathbf{1}_{\sigma}(k)\widehat{\Phi}(\sigma\setminus k)
= \mathbf{1}_{\Gamma\!_{k]}}(\sigma\setminus k)\mathbf{1}_{\sigma}(k)\widehat{\Phi}(\sigma\setminus k)
= \widehat{\mathfrak{a}_k^{\dag}\mathfrak{E}_k\Phi}(\sigma),
\end{equation*}
where equality $\mathbf{1}_{\Gamma\!_{k]}}(\sigma)\mathbf{1}_{\sigma}(k)=\mathbf{1}_{\Gamma\!_{k]}}(\sigma\setminus k)\mathbf{1}_{\sigma}(k)$ is used.
Thus $\mathfrak{E}_k\mathfrak{a}_k^{\dag} = \mathfrak{a}_k^{\dag}\mathfrak{E}_k$ holds.
Similarly, we can verify $\mathfrak{E}_k\mathfrak{a}_k = \mathfrak{E}_{k-1}\mathfrak{a}_k$.
\end{proof}

Combining Theorem~\ref{thr-3-12} with Proposition~\ref{prop-3-13}, we come to the next interesting result, which we call
the Clark-Ocone formula for generalized functionals of $Z$.

\begin{theorem}\label{thr-3-14}
For all generalized functional $\Phi\in \mathcal{S}^*(Z)$, it holds that
 \begin{equation}\label{eq-3-21}
   \Phi = \mathfrak{E}\Phi + \sum_{k=0}^{\infty}\mathfrak{a}_k^{\dag}\mathfrak{E}_{k-1}\mathfrak{a}_k\Phi,
 \end{equation}
where $\mathfrak{E}_{-1}=\mathfrak{E}$ and the series on the righthand side converges strongly in $\mathcal{S}^*(Z)$.
\end{theorem}

\begin{remark}
As mentioned above, $\partial_k$ and $\partial_k^*$ are the annihilation and creation operators on $\mathcal{L}^2(Z)$, respectively,
and $P_k = \mathbb{E}[\cdot\,|\, \mathcal{F}_k]$ is the conditional expectation operator on $\mathcal{L}^2(Z)$.
It can be verified that
\begin{equation}\label{eq-3-22}
\partial_k^* P_{k-1}\eta = Z_k P_{k-1}\eta,\quad \forall\, k\geq 0,\, \forall\,\eta\in \mathcal{L}^2(Z),
\end{equation}
where $P_{-1}=\mathbb{E}$ and $Z_k$ the $k$-component of the discrete-time normal noise $Z$.
Thus the Clark-Ocone formula (\ref{eq-1-1}) can be rewritten as the following form
\begin{equation}\label{eq-3-23}
  \xi = \mathbb{E}\xi + \sum_{k=0}^{\infty} \partial_k^* P_{k-1}\partial_k \xi,\quad \xi \in \mathcal{L}^2(Z),
\end{equation}
where the series on the righthand side converges in the norm of $\mathcal{L}^2(Z)$.
This observation justifies calling formula (\ref{eq-3-21}) the Clark-Ocone formula for generalized functionals of $Z$.
\end{remark}

\section{Applications}\label{sec-4}

In the final section, we show some applications of our Clark-Ocone formula.

For $p\geq 0$ and $\Phi$, $\Psi\in \mathcal{S}^*(Z)$, we define $\langle \Phi,\Psi\rangle_{-p}$ as
\begin{equation}\label{eq-4-1}
  \langle \Phi,\Psi\rangle_{-p}= \sum_{\sigma\in \Gamma} \lambda_{\sigma}^{-2p} \widehat{\Phi}(\sigma)\overline{\widehat{\Psi}(\sigma)}
\end{equation}
provided the series on the righthand side absolutely converges.
Note that, if $\Phi$, $\Psi\in \mathcal{S}_p^*(Z)$, then by Theorem~\ref{thr-2-8} the series in (\ref{eq-4-1}) absolutely converges,
hence $\langle \Phi,\Psi\rangle_{-p}$ makes sense, and in particular
\begin{equation}\label{eq-4-2}
  \langle \Phi,\Phi\rangle_{-p} = \|\Phi\|_{-p}^2.
\end{equation}

\begin{definition}\label{def-4-1}
For generalized functionals $\Phi$, $\Psi\in \mathcal{S}^*(Z)$, their $p$-covariant $\mathrm{Cov}_p(\Phi,\Psi)$, $p\geq0$, is defined as
\begin{equation}\label{eq-4-3}
  \mathrm{Cov}_p(\Phi,\Psi)= \langle \Phi-\mathfrak{E}\Phi,\Psi-\mathfrak{E}\Psi\rangle_{-p}
\end{equation}
provided the righthand side makes sense.
\end{definition}

By convention, $\mathrm{Var}_p(\Phi)\equiv\mathrm{Cov}_p(\Phi,\Phi)$ is called the $p$-variant of generalized functional $\Phi$.
Clearly, $\mathrm{Var}_p(\Phi)=\|\Phi-\mathfrak{E}\Phi\|_{-p}^2$ if $\Phi\in \mathcal{S}_p^*(Z)$.

\begin{theorem}\label{thr-4-1}
Let $\Phi$, $\Psi\in \mathcal{S}_p^*(Z)$ for some $p\geq 0$. Then
their $p$-covariant $\mathrm{Cov}_p(\Phi,\Psi)$ makes sense, and moreover
\begin{equation}\label{eq-4-4}
  \mathrm{Cov}_p(\Phi,\Psi)
=\sum_{k=0}^{\infty}\big\langle \mathfrak{E}_k\mathfrak{a}_k^{\dag}\mathfrak{a}_k\Phi, \mathfrak{E}_k\mathfrak{a}_k^{\dag}\mathfrak{a}_k\Psi\big\rangle_{-p}.
\end{equation}
\end{theorem}

\begin{proof}
By Theorem~\ref{thr-2-8}, the series on the righthand side of (\ref{eq-4-1}) converges absolutely.
On the other hand, by Theorem~\ref{thr-3-12}, we have
\begin{equation*}
\begin{split}
  \mathrm{Cov}_p(\Phi,\Psi)
  &= \sum_{\sigma\in \Gamma} \lambda_{\sigma}^{-2p} \widehat{\Phi-\mathfrak{E}\Phi}(\sigma)\overline{\widehat{\Psi-\mathfrak{E}\Psi}(\sigma)}\\
  &= \sum_{\sigma\in \Gamma} \lambda_{\sigma}^{-2p}
      \Big[\sum_{k=0}^{\infty}\mathbf{1}_{\Gamma\!_{k]}}(\sigma)\mathbf{1}_{\sigma}(k)\widehat{\Phi}(\sigma)\Big]
      \Big[\sum_{k=0}^{\infty}\mathbf{1}_{\Gamma\!_{k]}}(\sigma)\mathbf{1}_{\sigma}(k)\overline{\widehat{\Psi}(\sigma)}\,\Big],
\end{split}
\end{equation*}
which together with the fact
\begin{equation*}
  \mathbf{1}_{\Gamma\!_{j]}}(\sigma)\mathbf{1}_{\sigma}(j)\mathbf{1}_{\Gamma\!_{k]}}(\sigma)\mathbf{1}_{\sigma}(k)=0,\quad
j\neq k,\, j,\, k\geq 0,\, \sigma\in \Gamma
\end{equation*}
gives
\begin{equation*}
\begin{split}
  \mathrm{Cov}_p(\Phi,\Psi)
    &= \sum_{\sigma\in \Gamma} \lambda_{\sigma}^{-2p}
      \sum_{k=0}^{\infty}\mathbf{1}_{\Gamma\!_{k]}}(\sigma)\mathbf{1}_{\sigma}(k)\widehat{\Phi}(\sigma)
      \overline{\widehat{\Psi}(\sigma)}\\
    &= \sum_{k=0}^{\infty}\sum_{\sigma\in \Gamma} \lambda_{\sigma}^{-2p} \big[\mathbf{1}_{\Gamma\!_{k]}}(\sigma)\mathbf{1}_{\sigma}(k)\widehat{\Phi}(\sigma)\big]
        \big[\mathbf{1}_{\Gamma\!_{k]}}(\sigma)\mathbf{1}_{\sigma}(k)\overline{\widehat{\Psi}(\sigma)}\big]\\
   &= \sum_{k=0}^{\infty}\sum_{\sigma\in \Gamma} \lambda_{\sigma}^{-2p}\widehat{\mathfrak{E}_k\mathfrak{a}_k^{\dag}\mathfrak{a}_k\Phi}(\sigma)
       \overline{\widehat{\mathfrak{E}_k\mathfrak{a}_k^{\dag}\mathfrak{a}_k\Psi}(\sigma)}\\
  &=  \sum_{k=0}^{\infty}\big\langle \mathfrak{E}_k\mathfrak{a}_k^{\dag}\mathfrak{a}_k\Phi, \mathfrak{E}_k\mathfrak{a}_k^{\dag}\mathfrak{a}_k\Psi\big\rangle_{-p}.
\end{split}
\end{equation*}
This completes the proof.
\end{proof}

Theorem~\ref{thr-4-1} sets up covariant identities for generalized functionals of $Z$. The next theorem then gives meaningful
upper bounds to variants of generalized functionals of $Z$.

\begin{theorem}\label{thr-4-2}
Let $\Phi\in \mathcal{S}_p^*(Z)$ for some $p\geq 0$. Then its
$p$-variant $\mathrm{Var}_p(\Phi)$ makes sense, and moreover
\begin{equation}\label{}
  \mathrm{Var}_p(\Phi)
\leq \sum_{k=0}^{\infty}\big\|\mathfrak{a}_k^{\dag}\mathfrak{a}_k\Phi\|_{-p}^2.
\end{equation}
\end{theorem}

\begin{proof}
By Theorems~\ref{thr-3-2}, \ref{thr-3-4} and \ref{thr-3-10}, we know that $\mathfrak{E}_k\mathfrak{a}_k^{\dag}\mathfrak{a}_k\Phi$ belongs to $\mathcal{S}_p^*(Z)$
and
\begin{equation*}
\|\mathfrak{E}_k\mathfrak{a}_k^{\dag}\mathfrak{a}_k\Phi\|_{-p}\leq \|\mathfrak{a}_k^{\dag}\mathfrak{a}_k\Phi\|_{-p},\quad k\geq 0.
\end{equation*}
This together with (\ref{eq-4-2}) and (\ref{eq-4-4}) yields
\begin{equation*}
  \mathrm{Var}_p(\Phi)
=\sum_{k=0}^{\infty} \|\mathfrak{E}_k\mathfrak{a}_k^{\dag}\mathfrak{a}_k\Phi\|_{-p}^2
\leq \sum_{k=0}^{\infty}\big\|\mathfrak{a}_k^{\dag}\mathfrak{a}_k\Phi\|_{-p}^2.
\end{equation*}
This completes the proof.
\end{proof}

A sequence $u=(u_k)$ of generalized functionals in $\mathcal{S}^*(Z)$ is said to be $(\mathfrak{E}_k)$-predictable if
\begin{equation}\label{}
  u_k= \mathfrak{E}_{k-1}u_k,\quad k\geq 0.
\end{equation}
It is said to be $(\mathfrak{a}_k^{\dag})$-integrable if the series $\sum_{k=0}^{\infty}\mathfrak{a}_k^{\dag}u_k$ converges
strongly in $\mathcal{S}^*(Z)$. In that case, we call $\sum_{k=0}^{\infty}\mathfrak{a}_k^{\dag}u_k$ the generalized stochastic integral
of $u$ with respect to $(\mathfrak{a}_k^{\dag})$ and write
\begin{equation}\label{}
  \mathfrak{I}(u)= \sum_{k=0}^{\infty}\mathfrak{a}_k^{\dag}u_k.
\end{equation}

\begin{theorem}\label{thr-4-3}
Let $\Phi\in \mathcal{S}^*(Z)$. Then the sequence $u=(\mathfrak{E}_{k-1}\mathfrak{a}_k \Phi)_{k\geq 0}$ of
generalized functionals in $\mathcal{S}^*(Z)$ is $(\mathfrak{E}_k)$-predictable and
$(\mathfrak{a}_k^{\dag})$-integrable, and moreover
\begin{equation}\label{eq-4-8}
  \Phi = \mathfrak{E}\Phi+ \mathfrak{I}(u).
\end{equation}
\end{theorem}

\begin{proof}
This is an immediate consequence of Theorem~\ref{thr-3-14}.
\end{proof}

\begin{remark}\label{rem-4-1}
A generalized functional of $Z$, or in other words, a generalized functional in $\mathcal{S}^*(Z)$,
can be interpreted as a generalized random variable on the probability space $(\Omega, \mathcal{F}, P)$.
Accordingly, a sequence of generalized functionals of $Z$ can be viewed as a generalized stochastic process.
Theorem~\ref{thr-4-3} then shows that each generalized random variable on $(\Omega, \mathcal{F}, P)$
can be represented as the generalized stochastic integral of an $(\mathfrak{E}_k)$-predictable generalized stochastic process
with respect to $(\mathfrak{a}_k^{\dag})$.
\end{remark}

\section*{Acknowledgement}

This work is supported by National Natural Science Foundation of China (Grant No. 11461061).

\end{document}